\documentclass[11pt]{article}
\usepackage{amsmath}
\usepackage{mathrsfs}
\usepackage{amsfonts}
\usepackage{amsmath,amsthm}
\usepackage{amsmath,amssymb,amsthm,latexsym}
\usepackage{amscd}
\usepackage[all]{xy}
\usepackage{hyperref}

\newtheorem{theorem}{Theorem}[section]
\newtheorem{proposition}[theorem]{Proposition}
\newtheorem{conjecture}[theorem]{Conjecture}
\newtheorem{definition}[theorem]{Definition}
\newtheorem{corollary}[theorem]{Corollary}
\newtheorem{lemma}[theorem]{Lemma}

\theoremstyle{remark}
\newtheorem{remark}[theorem]{Remark}

\def\<{\langle}
\def\>{\rangle}

\begin{document}
\title{\bf{A Poincar\'{e}-type inequality and a related eigenvalue problem}}

\author{Nan Ye%
  \thanks{School of Mathematical Sciences, Peking University,
 Beijing 100871, People's Republic of China. \texttt{yen@pku.edu.cn}. }
\and Xiang Ma%
  \thanks{LMAM, School of Mathematical Sciences, Peking University,
 Beijing 100871, People's Republic of China. \texttt{maxiang@math.pku.edu.cn}, Fax:+86-010-62751801. Corresponding author.}}

\date{\today}

\maketitle

\begin{center}
{\bf Abstract}
\end{center}

Given a smooth positive function $f$ defined on the unit circle satisfying a simple condition, we obtain a Poincar\'{e}-type inequality for an arbitrary function $u$ whose weighted average with respect to $f$ is zero. The proof uses Fenchel's theorem about the total curvature of closed space curves in an essential way. Next we consider the generalization of this result to higher dimensional closed Riemannian manifold and reduce it to an eigenvalue problem. Finally, we point out that even though such Poincar\'{e}-type inequality still holds, the best constant $\lambda_1(f)$ might be different from the first eigenvalue $\lambda_1$ by constructing explicit examples on the standard spheres and flat tori.

\hspace{2mm}


{\bf MSC(2000):\hspace{2mm} }

\section{Introduction}

In this paper, we report a Poincar\'e-type inequality for functions defined on $S^1$ as below.

\begin{theorem}\label{thm-main}
Let $f(\theta)$ be a fixed positive $2\pi$-period $C^1$ function on $\mathbb{R}$ such that
\begin{equation}\label{eq-orthogonal}
\int_0^{2\pi}f(\theta)\cos\theta d\theta=\int_0^{2\pi}f(\theta)\sin\theta d\theta=0.
\end{equation}
Then for any $2\pi$-period $C^1$ function $\phi$ satisfying \begin{equation}\label{eq-average}
\int_0^{2\pi}f(\theta)\phi(\theta)d\theta=0,
\end{equation}
we have
\begin{equation}\label{eq-main}
\int_0^{2\pi}\phi(\theta)^2d\theta\leq
\int_0^{2\pi}(\phi'(\theta))^2d\theta.
\end{equation}
The equality is attained only when
$\phi(\theta)=a\cos\theta+b\sin\theta$ for real constants $a,b$.
\end{theorem}

To the best of our knowledge, this inequality seems to be a new result. It is an unexpected byproduct of our study of the global geometry of closed curves in 3-dimensional (Euclidean and Lorentz) space, and the proof follows from the famous Fenchel's theorem (see Section~2 and {\cite{MaYe}}).

We tried to find a purely analytical proof without success. This situation raises two tantalizing questions: What is the significance of this analytical result? Can we generalize it to functions on other higher dimensional Riemannian manifolds?

Notice that if we take $f\equiv 1$ identically in the theorem above, then \eqref{eq-orthogonal} holds true automatically, and our result reduces to the classical Wirtinger's inequality. The natural generalization of the latter to higher dimensional spaces is Poincar\'e inequality \cite{SchoenYau}, which states that for a given closed Riemannian manifold $(M^n,g)$, there exists a universal constant $C>0$ such that for any $C^1$ functions $u$ satisfying
\begin{equation}\label{eq-average}
\int_M u d\mu =0,
\end{equation}
there always holds
\begin{equation}
C\int_M  u^2 d\mu ~\leq~ \int_M|\nabla u|^2 d\mu.
\end{equation}
In particular, $C$ can be chosen to be $\lambda_1$, the first eigenvalue of the standard Laplacian on $(M^n,g)$.

Based on this observation, we propose the following conjecture as a parallel generalization of our Poincar\'e-type inequality \eqref{eq-main}:
\begin{conjecture}\label{conj-f}
Let $(M^n,g)$ be a compact Riemannian manifold with $\partial M=\emptyset$. Let $\lambda_1$ be the first eigenvalue of the Laplace operator $\Delta$ on $(M^n,g)$ and $\mathscr{U}_1$ be the first eigenspace of $\Delta$. Let $f:M\rightarrow\mathbb{R}^+$ be a fixed smooth function satisfying $f\perp \mathscr{U}_1$. Then for any smooth function $\phi:M\rightarrow\mathbb{R}$ satisfying $\int_M f\phi d\mu=0$, we have
\begin{equation}\label{eq-poincaretype}
\lambda_1\int_M\phi^2d\mu\leq\int_M|\nabla\phi|^2d\mu.
\end{equation}
\end{conjecture}

If we denote the Rayleigh quotient by $R(u)=\int|\nabla u|^2d\mu/\int u^2 d\mu$, then this problem is equivalent to showing
\begin{equation}\label{}
\lambda_1\le
    \inf_{u\in f^\perp}R(u).
    \end{equation}
The right hand side could be viewed as a constrained variation problem, whose corresponding Euler-Lagrange equation is
\begin{equation}\label{}
\triangle u+R(u)u\equiv0~~(\mathrm{mod} ~f).
 \end{equation}
Thus we make the following definitions.

\begin{definition}\label{def-eigenvalue}
Let $(M^n,g)$ be a closed Riemannian manifold with a fixed positive function $f$ defined on it.
A constant $\lambda$ is called an \emph{$f$-eigenvalue} if there exist some non-trivial smooth function $u\in f^\perp$ (with respect to the $L^2$ norm on the function space over $(M^n,g)$) such that
\begin{equation}
\triangle u\equiv-\lambda u ~~(\mathrm{mod} ~f).
\end{equation}
This corresponding $u$ is called the related $f$-eigenfunction.
\end{definition}

As usual, such $f$-eigenvalues form a spectrum $\{\lambda_k(f)\}_{k=1}^{+\infty}$, called \emph{the $f$-spectrum}:
\begin{equation}
0\leq\lambda_1(f)\leq\lambda_2(f)\leq\cdots\leq\lambda_k(f)
\leq\cdots
\end{equation}
They share many similar properties as the usual eigenvalues $\{\lambda_k\}_{k=1}^{+\infty}$ \cite{Chavel}. In particular, the min-max principle still holds true. This enables us to generalize the classical Cauchy interlace theorem directly (which is also known as Poincar\'e separation theorem; for reference, please see \cite{Dancis, Davis, Hwang}). As part of the results, generally we have
\[\lambda_1(f)=
    \inf_{u\in f^\perp}R(u)\leq\lambda_1.\]
We observe that:

1) $f\perp \mathscr{U}_1$ is
a necessary condition for $\lambda_1(f)=\lambda_1$.

2) Conjecture~\ref{conj-f} means that $f\perp \mathscr{U}_1$ is also sufficient to derive $\lambda_1(f)=\lambda_1$.\\

Although we obtained the interlace theorem and some other results, our attempt to prove Conjecture~\ref{conj-f} all failed. Finally this conjecture turns out to be not true.\\

\textbf{Conclusion}: There exist counterexamples to Conjecture~\ref{conj-f} on the spheres $S^n (n\ge 2)$ (with the standard metric) and on the square torus $T^2$ (with the flat metric). \\

The basic idea for the construction of these counterexamples is:

1. Choose $M$ from the simplest Riemannian manifolds with many symmetries whose eigenfunctions and eigenvalues are known explicitly.

2. Using the Ansatz that $\phi=\phi_2+\epsilon$, i.e., a combination of some second eigenfunction $\phi_2$ and the constant function. The constant $\epsilon$ is chosen suitably so that $\phi$ changes sign on $M$, at the same time the Rayleigh quotient $R(\phi)=\int|\nabla \phi|^2d\mu/\int \phi^2 d\mu$ is smaller than $\lambda_1$.

3. Show that there exists a positive even function $f$ orthogonal to the first eigenspace $\mathscr{U}_1$ (which consist of only odd functions) and $\phi$ (which changes its sign).

Note that these counterexamples are only with respect to some special functions $f$. If the function $f$ are given arbitrarily, the answer might be different. For example, the Poincar\'e-type inequality is still true when $f$ is a very small perturbation of the constant function. In view of these facts, the following problem seems to be subtle and interesting.\\

\textbf{Problem}: Is there any closed Riemannian manifold $(M^n, g)$ other than $S^1$ such that the conclusion of Conjecture~\ref{conj-f} holds true for any positive function $f\perp \mathscr{U}_1$?\\

We guess the answer is negative. If so, then $S^1$ can be characterized as the unique Riemannian manifold on which the generalized Poincar\'e inequality \eqref{eq-poincaretype} holds true under the weakest conditions in Conjecture~\ref{conj-f}.

This paper is organized as follows. In Section~2 we prove Theorem~\ref{thm-main}. The formulation of the eigenvalue problem and the estimation of the $f$-eigenvalues are presented in Section~3 .
We give the construction of the counterexamples to Conjecture~\ref{conj-f} in the last section.\\

\textbf{Acknowledgement}~~
This work is supported by the Project 11471021 of the
National Natural Science Foundation of China.

\section{The proof to the main theorem}

This section establishes the Poincar\'e type inequality for functions on $S^1$. The inequality is obtained as a corollary of Fenchel's theorem \cite{Chern, docarmo} on the total curvature of closed curves in $\mathbb{R}^3$. Then we will characterize the equality case in the second subsection.

\subsection{The inequality follows from Fenchel's theorem}

By the assumption of Theorem~1.1, we are given a fixed $2\pi$-period $C^1$ function $f(\theta)$ on $\mathbb{R}$ satisfying:
1) $\int_0^{2\pi}f(\theta)\cos\theta d\theta=\int_0^{2\pi}f(\theta)\sin\theta d\theta=0;$
2) $f$ is a positive function.
Set
\begin{equation}\label{eq-curve}
\gamma_\sigma(\theta)=\int_0^\theta f(\theta)(\cos\theta,\sin\theta,\sigma\phi(\theta))d\theta,
\theta\in[0,2\pi],\sigma\in(-\epsilon,\epsilon).
\end{equation}
Then $\gamma_\sigma$ form an one-parameter family of $C^2$ closed curves in $\mathbb{R}^3$ (which differ from each other by a vertical scaling). By standard computation we find their total curvatures as
\begin{equation}\label{eq-curvature}
L(\sigma)\triangleq\int_{\gamma_\sigma}kds
=\int_0^{2\pi}\frac{\sqrt{1+\sigma^2\phi^2+\sigma^2(\phi')^2}}
{1+\sigma^2\phi^2}d\theta.
\end{equation}
Fenchel's theroem says $L(\sigma)\geq2\pi=L(0)$. Thus there should hold $L'(0)=0$ (which is obviously true for our $\gamma_\sigma$), and $L''(0)\geq0$. A direct computation shows that
\begin{equation}
L''(0)=\int_0^{2\pi}(\phi'(\theta))^2d\theta-\int_0^{2\pi}
\phi(\theta)^2d\theta,
\end{equation}
so the inequality follows immediately.

\subsection{The equality case}
We have verified the inequality as above. To determine the equality case, which is more involved and subtler, we need to use this inequality to prove two technical lemmas beforehand.

Before giving the statement of the lemmas, we observe that the positivity of $f$ is a crucial assumption in Theorem~\ref{thm-main}. Without this assumption, the orthogonality condition $f\perp \mathscr{U}_1$ is not sufficient to establish the inequality \eqref{eq-main}. (An obvious counterexample is provided by choosing $f=\phi_2$ as the second eigenfunction and $\phi\equiv 1$.) Geometrically, the assumption $f>0$ guarantees that $\theta$ is a regular parameter of the closed space curves $\gamma_\sigma$ as above. Analytically, this requires that in the Fourier series of $f$, the first coefficient should not be too small compared with other coefficients. The following two lemmas provide a quantitative description of this restriction.

\begin{lemma}\label{lem-1}
Let $f(\theta)$ be a fixed $2\pi$-period $C^1$ function on $\mathbb{R}$ with
\begin{equation}\label{eq-orthogonal2}
\int_0^{2\pi}f(\theta)\cos\theta d\theta=\int_0^{2\pi}f(\theta)\sin\theta d\theta=0.
\end{equation}
Write
$f(\theta)=\frac{a_0}{2}+\sum_{n=2}^\infty(a_n\cos n\theta+b_n\sin n\theta).$
Suppose $f$ is positive. Then
\begin{equation}\label{eq-15}
\sum_{n=2}^\infty\frac{a_n^2+b_n^2}{n^2-1}\leq\frac{a_0^2}{2}.
\end{equation}
\end{lemma}
\begin{proof}
Consider an arbitrary $2\pi$-period $C^1$ function $\phi(\theta)$ orthogonal to $f$ and write
\begin{equation}\label{eq-16}
\phi(\theta)=\frac{c_0}{2}+\sum_{n=1}^\infty(c_n\cos n\theta+d_n\sin n\theta).
\end{equation}
The orthogonality implies
\begin{equation}\label{eq-17}
c_0=-\frac{2}{a_0}\sum_{n=2}^\infty(a_nc_n+b_nd_n).
\end{equation}
By the inequality of Theorem~\ref{thm-main} established in the previous subsection, there follows
\begin{equation}\label{eq-18}
\int_0^{2\pi}\phi(\theta)^2d\theta\leq\int_0^{2\pi}
(\phi'(\theta))^2d\theta.
\end{equation}
Substituting \eqref{eq-16} and \eqref{eq-17} into \eqref{eq-18}, we obtain
\begin{equation}\label{eq-19}
\frac{2}{a_0^2}\left[\sum_{n=2}^\infty(a_nc_n+b_nd_n)\right]^2= \frac{c_0^2}{2}\le \sum_{n=2}^\infty(n^2-1)(c_n^2+d_n^2).
\end{equation}
Consider the following two vectors in the $\ell^2$ space:
\begin{eqnarray}
\xi&=&(\frac{a_2}{\sqrt{2^2-1}},\frac{b_2}{\sqrt{2^2-1}},
\cdots,\frac{a_n}{\sqrt{n^2-1}},\frac{b_n}{\sqrt{n^2-1}},\cdots),
\label{eq-20}\\
\eta&=&(\sqrt{2^2-1}c_2,\sqrt{2^2-1}d_2,\cdots,\sqrt{n^2-1}c_n,
\sqrt{n^2-1}d_n,\cdots).
\label{eq-21}
\end{eqnarray}
Then \eqref{eq-19} is equivalent to
\begin{equation}\label{eq-22}
\frac{\langle\xi,\eta\rangle^2}{|\eta|^2}\leq\frac{a_0^2}{2}
\end{equation}
for arbitrary nonzero $\eta\in\ell^2$. On the other hand, the Cauchy-Schwarz inequality implies
\begin{equation}\label{eq-23}
\sup_{\eta\neq0}\frac{\langle\xi,\eta\rangle^2}{|\eta|^2}=|\xi|^2=\sum_{n=2}^\infty
\frac{a_n^2+b_n^2}{n^2-1}.
\end{equation}
The conclusion follows directly from \eqref{eq-22} and \eqref{eq-23}.
\end{proof}

\begin{lemma}\label{lem-2}
Under the same assumption as Lemma~\ref{lem-1}, we have
\begin{equation}\label{eq-24}
\sum_{n=2}^\infty\frac{a_n^2+b_n^2}{n^2-1}<\frac{a_0^2}{2}.
\end{equation}
\end{lemma}
\begin{proof}
As a positive $C^1$ function, $f$ has a positive minimal value. So $f(\theta)-\frac{\epsilon}{2}$ is also a positive $2\pi-$period $C^1$ function for $\epsilon$ small enough.
Applying Lemma~\ref{lem-1} to $\tilde{f}(\theta)=f(\theta)-\frac{\epsilon}{2}=\frac{a_0-\epsilon}{2}
+\sum_{n=2}^\infty(a_n\cos n\theta+b_n\sin n\theta)$ instead of $f(\theta)$, we know
\begin{equation*}
\sum_{n=2}^\infty\frac{a_n^2+b_n^2}{n^2-1}
\leq\frac{(a_0-\epsilon)^2}{2}<\frac{a_0^2}{2}.\qedhere
\end{equation*}
\end{proof}

\begin{proof}[Proof of Theorem~\ref{thm-main}]
After establishing the inequality in Subsection~2.1, here we need only to determine the equality case. Assume that the equality holds true in \eqref{eq-main} for some function $\phi(\theta)$. Write out the Fourier series:
\begin{eqnarray}
\phi(\theta)&=&\frac{c_0}{2}+\sum_{n=1}^\infty(c_n\cos n\theta+d_n\sin n\theta),\label{eq-25}\\
f(\theta)&=&\frac{a_0}{2}+\sum_{n=2}^\infty(a_n\cos n\theta+b_n\sin n\theta).\label{eq-26}
\end{eqnarray}
Then
\begin{eqnarray}
\int_0^{2\pi}f(\theta)\phi(\theta)d\theta=0&\Leftrightarrow& c_0=-\frac{2}{a_0}\sum_{n=2}^\infty(a_nc_n+b_nd_n),
\label{eq-27}\\
\int_0^{2\pi}\phi(\theta)^2d\theta=\int_0^{2\pi}(\phi'(\theta))^2d
\theta&\Leftrightarrow&\sum_{n=2}^\infty(n^2-1)(c_n^2+d_n^2)
=\frac{c_0^2}{2}.\label{eq-28}
\end{eqnarray}
By \eqref{eq-27} and \eqref{eq-28}, we have
\begin{equation}\label{eq-29}
\frac{a_0^2}{2}\sum_{n=2}^\infty(n^2-1)(c_n^2+d_n^2)
=\left[\sum_{n=2}^\infty(a_nc_n+b_nd_n)\right]^2.
\end{equation}
Using the same notations as in \eqref{eq-20} and \eqref{eq-21}, and invoking the Cauchy-Schwarz inequality, from \eqref{eq-29} we obtain
\begin{equation}
\frac{a_0^2}{2}|\eta|^2 =\langle\xi,\eta\rangle^2 \leq |\xi|^2|\eta|^2.
\end{equation}
On the other hand, we also have $|\xi|^2<\frac{a_0^2}{2}$ by Lemma~\ref{lem-2}.
So there must be $\eta=0$, i.e. $c_n=d_n=0$ for all $n\geq2$. Hence $c_0=0$ by \eqref{eq-27}. Finally we have $\phi(\theta)=c_1\cos\theta+d_1\sin\theta$ as desired.
\end{proof}

\begin{remark}
It is interesting to note that the condition of the equality can be deduced from the inequality \eqref{eq-main} itself. The key observation is that the stronger inequality \eqref{eq-24} can be deduced from the weaker \eqref{eq-15} by allowing a small perturbation of $f$ to $f-\frac{\epsilon}{2}$.
\end{remark}

Conjecture~\ref{conj-f} is a generalization of the results in this section to any closed Riemannian manifold $M$. The orthogonality assumptions $f\perp \cos\theta,\sin\theta$ and $f\perp \phi$ enable us to recover closed space curves in $\mathbb{R}^3$ when $M=S^1$. But for other Riemannian manifolds we lack such a geometric picture. Neither can we find a connection with geometric quantities like the total curvature. So we need new method to solve this problem. The next section is such an attempt.

\section{The $f$-eigenvalues and its estimation}

In this section we fix a closed Riemannian manifold $M$ and a positive function $f:M\to \mathbb{R}^+$.
In the introduction we have noticed that Conjecture~\ref{conj-f} is essentially about finding out the infimum of the Rayleigh quotient $R(u)$ where the function $u$ is restricted on the orthogonal complement of $f$ in the function space. Suppose  $u\in f^\perp$ is a critical point of this variational problem. Then for arbitrary $v\in f^\perp$, the $v$-variation of $R(u)$ is zero. By direct computation, that means
\begin{eqnarray*}
0=\delta_vR(u)&=&\left.\frac{d}{ds}\right|_{s=0}R(u+sv)
=\left.\frac{d}{ds}\right|_{s=0}\frac{\int_M|\nabla u+s\nabla v|^2d\mu}{\int_M(u+sv)^2d\mu}\\
&=&\frac{2\int_M\nabla u\cdot\nabla vd\mu}{\int_Mu^2d\mu}-2\frac{\int_M|\nabla u|^2d\mu}{(\int_Mu^2d\mu)^2}\int_Muvd\mu\\
&=&-\frac{2}{\int_Mu^2d\mu}\int_M[\triangle u+R(u)u]vd\mu,
~~~~\forall~v\in f^\perp.
\end{eqnarray*}
So the Euler-Lagrange equation is $\triangle u+R(u)u=cf$ for some real constant $c$. This justifies Definition~\ref{def-eigenvalue} given in the introduction, where the \emph{$f$-eigenvalue $\lambda$} and the \emph{$f$-eigenfunction $u$} is defined by the equality
\begin{equation}\label{eq-eigenvalue}
\triangle u\equiv-\lambda u ~~(\mathrm{mod}~f).
\end{equation}

\begin{proposition}
Let $u_1,u_2\in f^\perp$ be two $f$-eigenfunctions on $M$ with distinct $f$-eigenvalues $\lambda_1\ne\lambda_2$. Then $u_1,u_2$ are orthogonal to each other, i.e.,
$\int_Mu_1u_2d\mu=0.$
\end{proposition}
\begin{proof}
By the definition of the $f$-eigenfunctions,
$0=\int_M\triangle u_1u_2d\mu-\int_M u_1\triangle u_2d\mu
=(\lambda_2-\lambda_1)\int_Mu_1u_2d\mu.$
\end{proof}

We can also define the $k$-th $f$-eigenvalues recursively by min-max method.
\begin{definition}\label{def-eigenvalue2}
Fix a closed Riemannian manifold $M$ and a positive function $f:M\to \mathbb{R}^+$. Set
\begin{equation}
\lambda_1(f)=\inf_{u\in f^\perp}R(u),
\end{equation}
and let $\phi_{1,f}$ be the corresponding $f$-eigenfunction.
Suppose that $\{\lambda_i(f)\}_{i=1}^k$ have been defined, and $\{\phi_{i,f}\}_{i=1}^k$ are the corresponding $f$-eigenfunctions. Then set
\begin{equation}
\lambda_{k+1}(f)=\inf_{u\in \text{Span}\{f,\{\phi_{i,f}\}_{i=1}^k\}^\perp}R(u).
\end{equation}
Let $\phi_{k+1,f}\in \mathrm{Span}\{f,\{\phi_{i,f}\}_{i=1}^k\}^\perp\ne 0$ to be a nonzero minimizer. This defines the $f$-Spectrum:
\begin{equation}
0\leq\lambda_1(f)\leq\lambda_2(f)\leq\cdots\leq\lambda_k(f)
\leq\cdots
\end{equation}
\end{definition}

Obviously, the $f$-eigenvalues are exactly the usual eigenvalues when $f$ is a constant. In the general case, we want to know the relationships between these two kinds of eigenvalues.
\begin{proposition}\label{prop-interlace0}
Under the assumptions and notations as above, we have
\begin{equation}\label{eq-35}
\lambda_1(f)\geq\lambda_1\frac{(\int_{M}fd\mu)^2}{\int_{M}d\mu\int_{M}f^2d\mu}>0.
\end{equation}
\end{proposition}

\begin{proof}
Denote $f$ as $f=f_0+\tilde{f}$, where $f_0=\frac{1}{|M|}\int_Mfd\mu>0$ and $\int_{M}\tilde{f}d\mu=0$.

For any $u\perp f$, written $u=u_0+\tilde{u}$, where $u_0=\frac{1}{|M|}\int_Mud\mu$. Then $u\perp f$ implies
\begin{equation}
\int_{M}\tilde{f}\tilde{u}d\mu=-f_0\int_Mud\mu.
\end{equation}
By Cauchy-Schwarz inequality, we have
\begin{eqnarray}
f_0^2|M|(\int_Mu^2d\mu-\int_M\tilde{u}^2d\mu)=f_0^2[\int_Mud\mu]^2
=[\int_{M}\tilde{f}\tilde{u}d\mu]^2\leq\int_{M}\tilde{f}^2d\mu\int_{M}\tilde{u}^2d\mu.
\end{eqnarray}
Since $\int_{M}\tilde{u}d\mu=0$, it gives $\lambda_1\int_{M}\tilde{u}^2d\mu\leq\int_{M}|\nabla\tilde{u}|^2d\mu$ by definition. Then
\begin{equation}
f_0^2|M|\int_{M}u^2d\mu\leq\frac{1}{\lambda_1}\int_{M}f^2d\mu\int_{M}|\nabla\tilde{u}|^2d\mu=\frac{1}{\lambda_1}\int_{M}f^2d\mu\int_{M}|\nabla u|^2d\mu.
\end{equation}
It shows that $R(u)\geq\lambda_1\frac{(\int_{M}fd\mu)^2}{|M|\int_{M}f^2d\mu}$ for any $u\in f^\perp$, then the desired result is given immediately by taking the infimum on $R(u)$.
\end{proof}

\begin{proposition}\label{prop-interlace1}
Under the assumptions and notations as the above definition, we have
\begin{equation}\label{eq-39}
\lambda_1(f)\leq\lambda_1.
\end{equation}
Denote $\mathscr{U}_1$ as the first eigenspace of the Laplace operator
(here the first eigenvalue $\lambda_1$ might have mulitiplicity larger than $1$). Then $\lambda_1(f)<\lambda_1$ if $f$ is NOT orthogonal to $\mathscr{U}_1$.
\end{proposition}

\begin{proof}
Let $\phi_1$ be an eigenfunction of the first eigenvalue $\lambda_1$. Then there exists some constant $c$ such that $(\phi_1+c)\perp f$. Note that
\begin{equation}\label{eq-40}
R(\phi_1+c)=\frac{\int_M|\nabla\phi_1|^2d\mu}
{\int_M(\phi_1+c)^2d\mu}=
\frac{\lambda_1\int_M\phi_1^2d\mu}{\int_M\phi_1^2d\mu+c^2|M|}
\leq\lambda_1.
\end{equation}
So we have
\begin{equation}
\lambda_1(f)=\inf_{u\in f^\perp}R(u)\leq R(\phi_1+c)\leq\lambda_1.
\end{equation}
If $f$ is not orthogonal to $\mathscr{U}_1$, then there exists an eigenfunction $\phi_1$ such that $\int_Mf\phi_1d\mu\neq0$. Then $c=-\frac{\int_Mf\phi_1d\mu}{\int_Mfd\mu}\neq0$. By \eqref{eq-40} we know $\lambda_1(f)<\lambda_1$.
\end{proof}

\begin{proposition}\label{prop-interlace2}
Let $\lambda_k(f)$ be the $k$-th $f$-eigenvalue, and $\lambda_k$ be the $k$-th eigenvalue of the Laplacian. Then we have
\begin{equation}
\lambda_k(f)\leq\lambda_k.
\end{equation}
\end{proposition}

\begin{proof}
Let $\{\phi_i\}_{i=1}^k$ be the first $k$ eigenfunctions (we do not count their multiplicities). Then there exist suitable non-trivial constants $\{c_i\}_{i=0}^k$ such that
\begin{equation*}
(c_0+\sum_{i=1}^kc_i\phi_i)\perp
\mathrm{Span}\{f,\{\phi_{i,f}\}_{i=1}^{k-1}\}.
\end{equation*}
Use $\phi=c_0+\sum_{i=1}^kc_i\phi_i$ as a test function for $\lambda_k(f)$.
The rest part of the proof is similar to that of the above proposition.
\end{proof}

\begin{proposition}\label{prop-interlace3}
Notations as above. We have
\begin{equation}
\lambda_k(f)\geq\lambda_{k-1}.
\end{equation}
\end{proposition}

\begin{proof}
Let $\{\phi_{i,f}\}_{i=1}^k$ be the first $k$ $f$-eigenfunctions. Then there exist some non-trivial constants $\{c_i\}_{i=1}^k$ such that
\begin{equation*}
\sum_{i=1}^kc_i\phi_{i,f}\perp
\mathrm{Span}\{1,\{\phi_i\}_{i=1}^{k-2}\}.
\end{equation*}
Taking $\sum_{i=1}^kc_i\phi_{i,f}$ to be a test function of $\lambda_{k-1}$, we can obtain the conclusion as in the previous two propositions.
\end{proof}

Combing Proposition~\ref{prop-interlace2} and Proposition~\ref{prop-interlace3}, we have proved the following
\begin{theorem}[Interlace Theorem]
\begin{equation}
0<\lambda_1(f)\leq\lambda_1\leq\lambda_2(f)\leq
\lambda_2\leq\cdots\leq\lambda_k(f)\leq\lambda_k
\leq\cdots\rightarrow+\infty.
\end{equation}
\end{theorem}

\begin{corollary}
Let $f,h:M^n\rightarrow\mathbb{R}^+$ be positive functions on $M^n$. Then $\lambda_k(f)\leq\lambda_{k+1}(h)$.
\end{corollary}

Below is a sharp estimation of $\lambda_k(f)$.

\begin{proposition}\label{prop-sharp}
Let $M$ be a closed Riemannian manifold and $f:M\to \mathbb{R}^+$ a positive function. Using the same notations as above, we have
\begin{equation}
\lambda_k\leq\lambda_k(f)+R(f).
\end{equation}
\end{proposition}

\begin{proof}
Let $\{\phi_{i,f}\}_{i=1}^k$ be the first $k$ $f$-eigenfunctions (orthogonal to each other). There exist some non-trivial constants $\{c_i\}_{i=0}^k$ such that
\begin{equation*}
c_0f+\sum_{i=1}^kc_i\phi_{i,f}\perp\text{Span}\{1,\{\phi_i\}
_{i=1}^{k-1}\}.
\end{equation*}
As a preparation, by Cauchy-Schwarz inequality we obtain
\begin{eqnarray}
(||c_0\nabla f||+||\sum_{i=1}^kc_i\nabla\phi_{i,f}||)^2
=(||c_0f||\cdot\frac{||\nabla f||}{||f||}+||\sum_{i=1}^kc_i\phi_{i,f}||
\frac{||\sum_{i=1}^kc_i\nabla\phi_{i,f}||}
{||\sum_{i=1}^kc_i\phi_{i,f}||})^2&& \notag\\
\leq(||c_0f||^2+||\sum_{i=1}^kc_i\phi_{i,f}||^2)(\frac{||\nabla f||^2}{||f||^2}+\frac{||\sum_{i=1}^kc_i\nabla\phi_{i,f}||^2}
{||\sum_{i=1}^kc_i\phi_{i,f}||^2}).&&
\end{eqnarray}
So we have
\begin{eqnarray*}
\lambda_k\leq R(c_0f+\sum_{i=1}^kc_i\phi_{i,f})
&=&\frac{||c_0\nabla f+\sum_{i=1}^kc_i\nabla\phi_{i,f}||^2}
{||c_0f+\sum_{i=1}^kc_i\phi_{i,f}||^2}\\
&\leq&\frac{(||c_0\nabla f||+||\sum_{i=1}^kc_i\nabla\phi_{i,f}||)^2}
{||c_0f||^2+||\sum_{i=1}^kc_i\phi_{i,f}||^2}\\
&\leq&\frac{||\nabla f||^2}{||f||^2}+\frac{||\sum_{i=1}^kc_i\nabla\phi_{i,f}||^2}
{||\sum_{i=1}^kc_i\phi_{i,f}||^2}\\
&=&R(f)+\frac{\sum_{i=1}^kc_i^2\lambda_i(f)||\phi_{i,f}||^2}
{\sum_{i=1}^kc_i^2||\phi_{i,f}||^2}\\
&\leq&R(f)+\lambda_k(f).\qedhere
\end{eqnarray*}
\end{proof}

\section{Counterexamples to the conjecture}
Despite our effort, the estimation of $\lambda_1(f)$ we obtained in the previous section is not enough to prove Conjecture~\ref{conj-f}. Naturally one is led to doubt this conjecture and search for counterexamples. Soon after considering this possibility seriously we found such examples.

Before giving the details, let us explain the basic idea at first. Suppose there exists a counterexample to Conjecture~\ref{conj-f}, which is a triple $(M, f, \phi)$ consisting of a closed Riemannian manifold $(M,g)$, a positive function $f:M\rightarrow\mathbb{R}^+$, and a smooth function $\phi:M\rightarrow\mathbb{R}$ such that $f\perp \mathscr{U}_1$, $\phi\perp f$ and $\int_M|\nabla\phi|^2d\mu<\lambda_1\int_M\phi^2d\mu$.
We would like it to be as simple as possible. So

1. We consider $M$ to be the unit sphere $S^n$ or a flat torus
with their standard metrics, because they have many symmetries, and their eigenvalues $\{0<\lambda_1<\cdots\}$ and eigenfunctions $\{\phi_k\} (\phi_0\equiv 1)$ are known already. (Here we count the multiplicities of the eigenvalues. This convention is different from the previous section.) Then we can expand any function using the eigenfunctions.

2. We do not fix the positive function $f$ beforehand, partly because its positivity is hard to characterize in terms of its Fourier series. Instead we begin with $\phi=\sum a_k\phi_k$. Because the higher order eigenfunctions $\phi_k$ have positive contribution to the Rayleigh quotient $R(\phi)$, we would like most of the $a_k$ vanish so that $R(\phi)$ is smaller and the expression $\phi=\sum a_k\phi_k$ is simpler. The best choice is $\phi=\epsilon+\phi_2$ where the constant $\epsilon$ will be chosen later.

3. The positive function $f$ is found in the last step. To guarantee $f\perp \phi$, $\phi$ should change its sign on $M$, which is also a sufficient condition to guarantee the existence of a positive function $f$ orthogonal to $\phi$. So $\epsilon$ must be small enough. To have $f\perp \mathscr{U}_1$, it is sufficient for $f$ to be an even function, because on $S^n$ and square torus the first eigenspace is spanned by odd functions. By symmetry, we need only to find a positive
$f$ on a fundamental domain $U$ (a quadrant) of $M$ so that $\int_U f\phi d\mu=0$ and extend it to be an even function on $M$. \\

\noindent \textbf{Counterexample~1}.
Let $S^n=\{(x_1,x_2,\cdots,x_{n+1})\in \mathbb{R}^{n+1}|x_1^2+x_2^2+\cdots+x_{n+1}^2=1\}$ be the  sphere with the standard induced metric. It is well-known that $\lambda_1(S^n)=n, \lambda_2(S^n)=2n+2$,
and $\phi_2=x_1^2-x_2^2$ is one of the second eigenfunctions which takes value in $[-1,1]$.

Let $\phi=\epsilon+x_1^2-x_2^2$ on $S^n$. There should be $\epsilon\in(-1,1)$ so that $\phi$ changes sign. It is easy to see
\begin{equation}
R(\phi)<\lambda_1~
\Longleftrightarrow~
\left(\frac{\lambda_2}{\lambda_1}-1\right)
\frac{\int_{S^n}(\phi_2)^2d\mu}{\mathrm{Vol}(S^n)}<\epsilon^2.
\end{equation}
By direct computation,
\begin{equation}
\left(\frac{\lambda_2}{\lambda_1}-1\right)\frac{\int_{S^n}
(x_1^2-x_2^2)^2d\mu}{\mathrm{Vol}(S^n)}
=\frac{4(n+2)}{n(n+1)(n+3)}.
\end{equation}
When $n\geq2$, $c_n\triangleq\frac{4(n+2)}{n(n+1)(n+3)}<1$, thus we can take $\epsilon\in(\sqrt{c_n},1)$, which guarantees that $ R(\phi)<\lambda_1$ and $\phi$ changes its sign. By the analysis above, this shows that Conjecture~\ref{conj-f} does not hold for higher dimensional sphere $S^n$ when $n\ge 2$.\\

\noindent \textbf{Counterexample~2}.
For the square torus $T^2=S^1\times S^1$, we know $\lambda_1(T^2)=1, \lambda_2(T^2)=2$,
and $\phi_2=\cos x\cos y$ is a second eigenfunction taking value in $[-1,1]$.
Let $\phi=\epsilon+\cos x\cos y$ on $T^2$.
By direct computation,
\begin{equation}
\left(\frac{\lambda_2}{\lambda_1}-1\right)
\frac{\int_{T^2}(\phi_2)^2d\mu}{\mathrm{Vol}(T^2)}
=
\left(\frac{\lambda_2}{\lambda_1}-1\right)
\frac{\int_{0}^{2\pi}\int_{0}^{2\pi}
\cos^2 x\cos^2 y dxdy}{4\pi^2}
=\frac{1}{4}.
\end{equation}
Thus we need only to take $\epsilon\in(\frac{1}{2},1)$. Then
$ R(\phi)<\lambda_1$ and $\phi$ changes its sign. By the same reason as  above, we know that Conjecture~\ref{conj-f} is not true for the square torus.

\begin{remark}
Although our Conjecture~\ref{conj-f} is generally not true, by Proposition~\ref{prop-interlace0} one can see that when $u$ is orthogonal to a fixed function $f$ and $\int_M f d\mu\ne 0$, there exists some universal constant $C$ such that the Poincar\'e type inequality
\[
C\int_M  u^2 d\mu ~\leq~ \int_M|\nabla u|^2 d\mu.
\]
still holds.
\end{remark}

\end{document}